\documentclass[a4paper,abstracton]{scrartcl}
\usepackage[utf8]{inputenc}
\usepackage{amsfonts,amsthm,amssymb,amsmath}
\usepackage[inline]{enumitem}
\usepackage{color}
\usepackage{graphicx} 
\usepackage{authblk}

\newtheorem{thm}{Theorem}

\newtheorem{conj}[thm]{Conjecture}

\theoremstyle{remark}

\title{List distinguishing index of graphs}
\author{Jakub Kwaśny, Marcin Stawiski\footnote{ Corresponding author; stawiski@agh.edu.pl}
}

\affil{AGH University,\\ Faculty of Applied Mathematics, \protect\\al. Mickiewicza 30, 30-059 Krakow, Poland}

\begin{document}
\maketitle
\begin{abstract}
We say that an edge colouring \emph{breaks} an automorphism if some edge is mapped to an edge of a different colour. We say that the colouring is \emph{distinguishing} if it breaks every non-identity automorphism. We show that such colouring can be chosen from any set of lists associated to the edges of a graph $G$, whenever the size of each list is at least $\Delta-1$, where $\Delta$ is the maximum degree of $G$, apart from a few exceptions. This holds both for finite and infinite graphs. The bound is optimal for every $\Delta\ge 3$, and it is the same as in the non-list version.

\end{abstract}

\section{Introduction}

In 1977, Babai \cite{BAB} introduced a concept of \emph{distinguishing} vertex colourings, which are those preserved only by the identity automorphism. 
The minimum number of colours in a distinguishing vertex colouring of a graph $G$ is called the \emph{distinguishing number} of $G$, and it is denoted by $D(G)$. The analogous parameter for edge colourings, introduced in 2015 by Pil\'sniak and Kalinowski \cite{KP}, is called the \emph{distinguishing index} of $G$ and denoted by $D'(G)$.  These concepts lie on the borderland between graph theory and abstract algebra, as they naturally generalize to an arbitrary group action \cite{primitive}. Automorphism breaking also plays an important role in the quasipolynomial time algorithm of Babai \cite{babaiisomorphism} for the graph isomorphism problem.


In this paper, we study the list version of distinguishing edge colourings. For each edge $e\in E(G)$, let $L(e)$ be a set of colours available for that edge. We are asking for the minimum cardinal number $k$ such that for any set of lists of cardinality $k$ we can find a distinguishing edge colouring $c$ such that $c(e)\in L(e)$ for every edge $e\in E(G)$. We denote this minimum $k$ as $D'_l(G)$, and we call this parameter the list distinguishing index of $G$. Clearly, $D'_l(G) \geq D'(G)$ for any graph $G$. 

List colourings were introduced in 1980 by Erdős, Rubin, and Taylor \cite{erdos} for the problem of proper vertex colourings. There are known classes of graphs with arbitrary large difference between the chromatic number and the required size of lists. However, for proper edge colourings there is a famous List Colouring Conjecture \cite{bollobas-list-conjecture,haggkvist-list-conjecture} which states that any graph $G$ has a proper edge colouring from any lists of size $\chi'(G)$. 

The list variant of vertex distinguishing colourings was first studied by Ferrara, Flesch and Gethner \cite{ferrara1} in 2011, and they conjectured that the list-distinguishing number is the same as the distinguishing number for every finite graph. There are a few partial results towards this conjecture: for finite trees \cite{ferrara2}, finite interval graphs \cite{immel}, graphs with dihedral automorphism group \cite{ferrara1}, Cartesian product of two finite cliques \cite{furedi}, and Kneser graphs \cite{kneser}.

Motivated by the conjecture of Ferrara, Flesch and Gethner, and by the List Colouring Conjecture, we propose the following.
\begin{conj}\label{ourconjecture} Let $G$ be a connected, infinite or finite graph. Then $D'_l(G)=D'(G)$.
\end{conj}

In the paper, we aim to provide a general upper bound for connected graphs, both finite and infinite. These types of bounds are known for the distinguishing index. For finite graphs, Pil\'sniak \cite{P} in 2017 proved the following.

\begin{thm}[\cite{P}] \label{thm: P} Let $G$ be a connected, finite graph that is neither a symmetric nor a bisymmetric
tree. If the maximum degree of $G$ is at least $3$, then
$D'(G) \le \Delta(G)-1$
unless $G$ is $K_4$ or $K_{3,3}$.
\end{thm}

Later, Pil\'sniak and Stawiski \cite{PS} proved the same claim for infinite graphs.

\begin{thm}[\cite{PS}] \label{thm: PS} Let $G$ be a connected, infinite graph with finite maximum degree $\Delta\ge 3$.
Then $D'(G) \le \Delta-1$.
\end{thm}

We show that these two bounds also hold for the list version of the problem. Since the above two results are optimal, so is ours. In particular, it follows that $D'_l(G)=D'(G)$ for every subcubic connected graph.

The proof is divided into two parts. The first, major part contains a proof for graph with cycles, and then we separately check trees. In formulating the theorems, we exclude the same exceptional graphs as Pil\'sniak \cite{P}, so we describe them shortly in the last section.

\section{Graphs with a cycle}

From now on, we only consider edge colourings. In the proofs below, we skip the case where all the lists are identical, as this case follows from Theorems \ref{thm: P} and \ref{thm: PS}. However, we note that our approach would allow this case to be included, at the expense of complexity of the proofs. 

\begin{thm}\label{thm:cycle}
Let $G$ be a connected graph with maximum degree $\Delta \ge 3$ which is not a tree and not isomorphic to $K_{3,3}$, nor $K_4$. Then $D'_l(G)\leq \Delta-1$.
\end{thm}
\begin{proof} Let $G=(V,E)$ be a connected graph and $\Delta=\Delta(G)$ be its maximum degree. Assume that $G$ is not a tree and $G\not\in \{K_{3,3}, K_4\}$. Let $L = \{L(e)\}_{e\in E}$ be a set of lists, each of size $\Delta-1$. Denote $L(u)=\bigcup_{uv\in E} L_{uv}$ for any $u\in V$. 

First, consider the case when $\Delta$ is infinite. Since $G$ is connected, then $G$ must have exactly $\Delta$ edges. Hence, we can pick a different colour for each edge to obtain a distinguishing colouring with $\Delta = \Delta-1$ colours. For the rest of the proof, we shall assume that $\Delta$ is finite.

For each colour $i\in \bigcup_{e\in E} L(e)$, we consider a subgraph $H_i$ induced by all the edges $e$, such that $i\in L(e)$. If $H_i=G$, then we call such a subgraph trivial (we shall also sometimes say that the colour $i$ is trivial). If every $H_i$ is trivial, then we have a standard non-list colouring, which exists by Theorems \ref{thm: P} and \ref{thm: PS} (we use the assumptions that $G$ is not a tree, so it is not a symmetric nor a bisymmetric tree, and that $G\not\in \{K_{3,3}, K_4\}$). Therefore, we can assume that not every $H_i$ is trivial. 

We shall describe a greedy algorithm which iteratively chooses the colours of the edges of $G$ from the respective lists. The algorithm starts by colouring some starting subgraph $G_0$. All the edges of $G_0$ are coloured at this step, and this colouring is distinguishing for $G_0$. We shall  guarantee in the further course that $G_0$ is coloured uniquely,  which will cause $G_0$ to be fixed. Then, the algorithm processes the remaining vertices, one by one, and fixes each new vertex it has reached, i.e. any vertex that is incident to a coloured edge. 

\vspace{4mm}
\noindent\textsc{I. The starting subgraph}
\vspace{2mm}

We consider the following cases to select a suitable starting subgraph. This choice also affects the later colouring strategy, when we must avoid the colour pattern used on the starting subgraph.

\textbf{Case 1.}  There exists a colour $p$ such that $H_p$ is non-trivial and it contains a cycle.
We shall call this colour pink.

Let $C$ be an induced cycle in $H_p$. Since $H_p$ is non-trivial, it must contain a vertex $v$, in the same connected component of $H_p$ as $C$, which has an incident edge $vw$ outside $H_p$ (note that $w$ may be in $H_p$). By the choice of $v$, there exists a shortest path $R$ from $v$ to $C$ ending in a vertex $u$ of $C$ (and $u$ must be the only common vertex of $R$ and $C$). In particular, it may be the case that $v$ lies on $C$, then $R$ is trivial and $u=v$. Denote by $u^+$ a neighbour of $u$ on $C$. We define our starting subgraph $G_0$ as the subgraph induced by all the edges incident to the vertices of $C$ and $R$. 

We now specify a distinguishing colouring of the starting subgraph. We colour all the edges of $C$ and $R$ except $uu^+$ pink (this is possible since $C$ and $R$ are contained in $H_p$, so these edges have the colour pink on their lists) and assign $uu^+$ a colour other than pink; we shall call this colour blue. 
Let us consider all possible extensions of the current colouring to $G$ and all possible automorphisms of these coloured graphs that stabilise $C \cup R$. If none of these automorphisms acts non-trivially on it, then we only need to choose the colours for the edges not in $C$ nor $R$. For each vertex in $C\cup R$, we assign different colours other than pink to these edges. This can be done since each such vertex except $v$ has at most $\Delta-2$ neighbours outside $C \cup R$, and the lists have size $\Delta-1$. The vertex $v$ may have one more neighbour outside $C \cup R$ but it has also one incident edge with $\Delta-1$ colours different from pink.

If, on the other hand, there exists such an automorphism, it interchanges $v$ and $u^+$ and we must break it at this moment. Since $uu^+$ is an edge, then either $u=v$ or $v$ has a neighbour on $C$ different from its successor on $R$. This means that $v$ must have two neighbours in $G_0$. In this case, we would like to choose the colours on the edges incident to $v$ and $u^+$ such that these two vertices receive different palettes. But in this case, $v$ has at most $\Delta-2$ neighbours outside $C\cup R$ and $L(vw)$ does not contain the colour pink, so we have two possibilities for the last edge $vw$ we colour, which result in two different palettes of $v$. For the other vertices on $C\cup R$, including $u^+$, we do not have such freedom, but we can just succeed. Therefore, we first choose the colours for the edges incident to vertices other than $v$ (following the rule that for each vertex we choose different colours other than pink on the incident edges), and then to $v$ such that the palettes of $u^+$ and $v$ are different. This way, we break all the automorphisms of $G_0$. 

\textbf{Case 2.} For every colour $p$, the graph $H_p$ is either trivial or  $H_p$ is a forest.

Consider any induced cycle $C$ in $G$. If any edge of $C$ contained only non-trivial colours on its list, then all the lists in $G$ would be identical, and we have already assumed that this is not the case. Therefore, each edge of $C$ has a colour $p$ in its list such that $H_p$ is a forest. For any non-trivial colour $p$ on the lists of $C$, we can consider the longest path $P$ contained both in $C$ and $H_p$. Each such path $P$ is contained in a maximal path, a maximal ray, or a double ray in $H_p$, which we denote by $R$. If it is possible, that $R$ is not entirely contained in $C$, then we choose $p$, $C$ and $P$ accordingly (in other words, first we consider only the colours $p$ that have the longest $P$'s, and then we choose, if there is one, the one with $R\neq P$). We define our starting subgraph $G_0$ as the subgraph induced by all the edges incident to the vertices of $R$ and $C$. 

Denote by $u$ and $v$ the end-vertices of $P$. 
Let $R'$ be a maximal subpath or a subray of $R$ ending with $u$ or $v$ (without loss of generality, let it be $u$).
If $R'\neq P$ then we call the edge of $R'-P$ incident with the cycle $C$ \emph{the gadget} of $P$.

We start with colouring all the edges of $R'$ pink. The colouring of the rest of the edges of $G_0$ depends on the number of edges in $C-P$. 

If $C-P$ contains at least two edges, then we choose different colours for the edges $uu^-$ and $vv^+$, where $u^-$ and $v^+$ are the neighbours of $u$ and $v$, respectively, in $C-P$. These colours are different from pink by the maximality of $P$. We shall refer to these colours as blue and green, respectively. Next, for each vertex of $R'$, we choose different colours other than pink for the edges outside $R'$ (this is possible for the same reason as in Case 1). Then, we perform the following scheme, which we write down separately as it will be used again later.

\emph{Cycle colouring scheme.} We take two passes on the cycle, each time considering the vertices $u_1,\dots,u_k$ consecutively. First, we choose the colours for the edges of the cycle. If the current edge has the colour pink on its list, then we choose pink, unless this is the last edge we are colouring and this would result in exactly two pink paths of length $|P|$ and only two non-pink edges on the cycle. In this case, we choose a colour other than pink. If the current edge does not have the colour pink, then we choose any colour, unless the previous $|P|$ edges are pink, in which case we disallow blue or green, whichever would create the pink path of length $|P|$ surrounded by blue and green. Subsequently, we do a second pass and colour all the other edges adjacent to the vertices of the cycle.  Take a vertex $u_i$. We consider a few cases:
\begin{itemize} 
\item If $u_iu_{i+1}$ is pink, then we choose different colours other than pink for all the uncoloured edges incident to $u_i$. It is possible, since there are at most $\Delta-2$ such edges, and each of them has $\Delta-2$ colours other than pink on its list. 
\item If $u_iu_{i+1}$ has a colour other than pink, and the colour pink does not appear on all the lists of the uncoloured incident edges, then we forbid both pink and the colour of $u_iu_{i+1}$ on the incident edges and again choose different colours. Moreover, if $u_i$ is an end-vertex of a pink path of length $|P|$ on the cycle, then we forbid also blue or green (whichever does not appear on the other side of this path) on all the currently coloured edges. To argue that we can succeed, we observe that if blue or green is present on the list of $u_i u_{i+1}$, then this colour cannot appear on any of the lists of the incident edges outside $C$. This is because we would choose this colour to be called pink at the beginning, as it would yield a gadget. Furthermore, either there is no pink at all on the lists of the edges incident to $u_i$ (so we have only two forbidden colours) or there is one list with pink and one without it (which gives us an additional colour to choose from). 
\item If $u_iu_{i+1}$ has a colour other than pink, and all the incident edges have pink on their lists, then again we choose different colours other than pink and the colour of $u_iu_{i+1}$ on the incident edges. Moreover, if $u_i$ is an end-vertex of a pink path of length $|P|$ on the cycle, and we are forced to use blue or green (whichever does not appear on the other side of this path) somewhere on the edge incident to $u_i$, then we put this colour on $u_i u_{i+1}$. This may create a copy of a pink path of length $|P|$ surrounded by blue and green, but it will cause no problem due to an absence of a gadget (and $P$ must have a gadget, since the path just created could have one). Note that we have used the rule, that if an edge on the cycle has a colour other than pink, then there is no pink on its list. There might be one exception to this rule, but it does not concern us because this exception does not occur at the end of the pink path of length $|P|$, but rather of $|P|-1$. 
\end{itemize} 

If $C-P$ contains only one edge $uv$ (it must contain at least one, since $H_p$ is a forest) and $u$ and $v$ have degree at least three in $G$, then we choose different colours of edges incident to $u$ and $v$ so that the palettes of $u$ and $v$ are different. We shall refer to the colour of $uv$ as blue. If $R'\neq P$, then $u$ has two adjacent pink edges and $u$ only one, so they are already distinguished. Otherwise, by the maximality of $R$, none of the edges incident to $u$ and $v$ outside $C$ has pink in its list, so there are at least one and at most $\Delta-2$ neighbours of each of these vertices outside $C$ and we can choose two different palettes. Then we choose the colours of the remaining edges, again like in the second pass of the Cycle colouring scheme.

If $C-P$ contains only one edge $uv$ and $d(u)=d(v)=2$, then we recolour the edge $uu^+$, where $u^+$ is a neighbour of $u$ on $C$ other than $v$, with a new colour different from pink. We shall refer to this colour as blue. We choose a colour other than pink and blue for the edge $uv$ and call it green. Then we choose the colours of the remaining edges, like in the second pass of the Cycle colouring scheme.

\vspace{2mm}

Depending on what the starting subgraph $G_0$ looks like and on the chosen colouring, we shall avoid the specific patterns during the remaining part of the algorithm. This will guarantee that $G_0$ is stabilised and, given the colouring of $G_0$, also fixed. 

Note that, in fact, there are only two types of starting subgraph: either an induced cycle with all incident edges, or an induced cycle with an attached path or ray, with all incident edges. In both cases, all the edges in $G_0$ not contained in the cycle, path nor ray are assigned a colour other than pink. Let $k$ be the length of the cycle. We shall use the name \emph{gadget} not only for the edge defined in Case 2, but also for the analogous edge in Case 1 (i.e. the one on the non-trivial path $R$, incident to a vertex of $C$). Moreover, we shall refer to the pink path on the cycle in $G_0$ as $P$, regardless of whether it was formed in Case 1 or Case 2.

We will also reuse the Cycle colouring scheme during the next part. In Case 2, the scheme started from some specific pre-coloured cycle, but we have never used the fact, what this initial colouring looked like. The main property of this scheme is that it will never produce another pink path of length $|P|$ surrounded by green and blue, with or without a gadget (depending on the existence of a gadget in $G_0$). Therefore, we shall use it, starting with some other initial colourings. 

\vspace{4mm}
\noindent\textsc{II. The iterative procedure}
\vspace{2mm}

We shall now iteratively extend the set of \emph{reached} vertices, i.e. the ones with a coloured incident edge, starting from $G_0$. We shall execute the procedure until there are no uncoloured edges left. Let $A$ be the set of the automorphisms which stabilise $G_0$ and preserve the partial colouring we defined so-far. After each execution of the procedure, we shall guarantee that the following conditions are satisfied:
\begin{enumerate}[label = (A\arabic*)]
    \item Each reached vertex is fixed pointwise with respect to $A$.
    \item If a vertex $v\notin V(G_0)$ has a pink incident edge, and it is the only coloured edge incident to $v$, then this edge is not contained in any cycle of length $k$.
\end{enumerate}

Note that these conditions are satisfied for the initial colouring of $G_0$. 

The procedure starts by taking a reached vertex $v$ with the smallest distance from $G_0$, which has an uncoloured edge. We shall call the already coloured edges of $v$ as back edges, the uncoloured edges to the reached vertices as horizontal edges and the remaining ones as forward edges. If none of the forward edges of $v$ appear in any induced cycle of length $k$, then we simply colour each forward edge of $v$ with a different colour, avoiding pink if possible, and then each horizontal edge with an arbitrary colour other than pink. This is possible since there are at most $\Delta-1$ of such edges, and it fixes pointwise each newly coloured vertices, so the conditions (A1) and (A2) are fulfilled. 

If there is an induced cycle of length $k$ containing a forward edge of $v$, then we first check the following conditions:
\begin{enumerate}[label = (C\arabic*)]
    \item Each forward edge of $v$ appears on a cycle of length $k$.
    \item All the lists of the forward edges of $v$ are the same, and each of them contains pink.
    \item There are $\Delta-1$ forward edges.
\end{enumerate}

If any of these condition is not satisfied, then we can colour the forward edges with different colours either without using pink (C2 or C3) or we can use pink on the edge which does not appear in such cycle (C1). If, however, all these conditions hold, our further actions shall depend on the structure of $G_0$. Let $C'$ be a cycle of length $k$ that contains a forward edge of $v$. If $C'$ contains also the unique back edge of $v$, then this edge is not pink by (A2) and $C'$ has two fixed vertices by (A1). Therefore, we just need to realise the cycle colouring scheme from Case 2 and then $C'$ will be fixed pointwise as long as $G_0$ is stabilized. There is only one exception: if $G_0$ is a cycle with all edges except one coloured pink, and the cycle colouring scheme produced an identical copy of $G_0$, then we change the colour of the blue edge to any other (including pink).

Assume now that $C'$ is a cycle of length $k$ that contains two forward edges of $v$. We must ensure that the colouring of $C'$ will be different from the one in $G_0$, otherwise $G_0$ will not be stabilized. Therefore, we will again colour the whole $C'$ with all incident edges at once, along with the edges incident to $v$. We colour the forward edges of $v$ which are not in $C'$ with different colours other than pink and blue. Then we colour one of the edges on $C'$ incident to $v$ pink, and the other one with any colour other than pink and blue. 

This last choice may be impossible if $\Delta=3$ and the lists of both forward edges consist of exactly pink and blue. In this case, we colour both edges pink and continue to choose pink in both directions on $C'$, until possible. Then on each side, we colour one next edge (it may be the same one edge) so that at least one of them is not blue. Afterwards, we continue like for the other values of $\Delta$, depending on the structure of $G_0$.
\begin{itemize}
    \item If $G_0$ has no gadget, or $G_0$ has a gadget but the back edge of $v$ is not pink, then, we just execute the cycle colouring scheme on $C'$. Note, that the cycle colouring scheme does not produce gadgets, so the back edge of $v$ would be the only candidate for one. 
    \item If $G_0$ has a gadget and the back edge $e$ of $v$ is pink, then by the assumption of the procedure, the edge $e$ is not contained in any cycle of length $k$. Hence, if we follow the cycle colouring scheme, then the only gadget created in this step can be $e$. But the gadget in $G_0$ was always incident to a blue edge, and there is no blue edge incident to $v$, therefore we are safe to execute the cycle colouring scheme on $C'$.
\end{itemize}

By the colouring of the two edges on $C'$ incident to $v$, we broke all the automorphisms of $C'$, given that $v$ was fixed. This and the cycle colouring scheme guarantee that all the reached vertices are fixed pointwise, so (A1) is satisfied. Moreover, we used the colour pink only on the cycle $C'$ or on some forward edge of $v$ which does not belong to any cycle of length $k$. This gives us (A2). 

We are left to show that we did not create a second copy of $G_0$ throughout the iterative procedure. Assume otherwise, and denote by $C''$ the cycle isomorphic to the cycle in $G_0$. There must be a pink edge $xy$ contained in a pink path $P''$ of length $|P|$ on $C''$, surrounded by blue and green edges or one blue edge, and the edge $xy$ does not belong to $G_0$. Let us assume that $xy$ is the edge incident to a blue edge on $C''$. Consider the step of the procedure when this edge was coloured. In the procedure, we used the colour pink for an edge in a cycle of length $k$ only when we coloured a cycle $C'$. We used the cycle colouring scheme, where the only possibility to create a pink path of length $|P|$ surrounded by blue and green edges was if $P$ had a gadget. But we ensured that the only pink edges incident to $C'$ lie on $C'$ itself, except for the currently processed vertex $v$ which has no incident blue edges, and therefore we could not have created a gadget of $P''$. We could not have created a cycle of length $k$ with all pink edges except one blue, either, as any pink path of length $k-1$ would be contained in $C'$, and this cycle is induced. This contradiction allows us to conclude that $G_0$ is fixed after the procedure, hence also the whole graph $G$.
\end{proof}

\section{Trees}\label{sec:trees}

\begin{thm} \label{thm:trees}
Let $G$ be a tree with maximum degree $\Delta\geq 3$. Then either $G$ is a symmetric tree, $G$ is a bisymmetric tree, or $D'_l(G)\leq \Delta-1$.
\end{thm}
\begin{proof}

Like in the proof of Theorem \ref{thm:cycle} we can assume that $\Delta$ is finite. We shall choose one vertex $r$ and refer to it as the root.  We shall use the standard notation and for any vertex $u$, we shall call the incident edge on the unique path from $u$ to $r$ as the \emph{back edge}, and all other edges incident to $u$ as \emph{forward edges}. 

We call a colouring of $(G,r)$ the \emph{standard colouring} if every vertex except $r$ has all the forward edges coloured with distinct colours. We claim that any standard colouring which fixes $N[r]$ (a closed neighbourhood of $r$) is a distinguishing colouring of $G$. To see this, consider any vertex $u$ outside $N[r]$ (as the elements of $N[r]$ are already fixed). Then there is a unique path from $r$ to $u$ through a neighbour $v$ of $r$. Consider the last vertex $w$ on that path, starting from $r$, which is fixed. If $w\neq u$, then some automorphism maps one forward edge of $w$ to another. But this is impossible, since these two edges, by the assumption, have different colours. This means that $w=u$, and $u$ must be fixed.

The remaining of the proof will consist of a few cases where we shall find a suitable root vertex $r$ and a standard colouring of $(G,r)$ with the above property. Note that having the edges incident to $r$ coloured, it is straightforward to find a standard colouring of the graph, e.g. by considering the vertices of $G$ one by one, from those closest to $r$. We shall usually be doing a variation of such procedure, as we shall need some additional properties.

    

\textbf{Case 1.} There is no vertex of degree at least two, with all incident edges sharing the same list.

We choose an arbitrary vertex $r$ and we colour all its incident edges with different colours. Let pink be one of these colours. We colour the second end-vertex $v$ of this pink edge so that $r$ and $v$ have distinct palettes. This is possible, since the edges incident to $v$, by our assumption, have different lists. Finally, we colour the remaining edges of $G$ with a standard colouring without using pink. This is again possible by our assumption.

In the following cases, we shall assume that there is a vertex which is not a leaf, such that all its incident edges share the same palette. 

\textbf{Case 2.} $G$ contains a vertex $v$ such that $1<d(v)<\Delta$.

We take such a vertex as a root $r$ and colour all its incident edges with different colours. Then, we colour the remaining edges to get a standard colouring, with the condition that each vertex of degree $d(r)$, apart from $r$, has a distinct palette than that of $r$. 


\textbf{Case 3.} $G$ is the regular tree of degree $\Delta$.

Let $r$ be an arbitrary vertex of degree at least two, with all incident edges sharing the same list. We start by colouring all the edges incident to $r$ with the same colour, say pink. We shall ensure that $r$ is the only vertex with all incident pink edges. Then, we iteratively fix the remaining vertices. During each iteration, we fix possibly only one vertex, but we choose a colour for multiple edges.  

Let $v$ be a vertex which is not yet fixed, and is the closest to $r$ among all such vertices. Let $i$ be the smallest natural number such that all the currently coloured edges are contained in $B(r,i)$ (i.e. the ball of radius $i$ centred at $r$). We choose a vertex $w$ which is a descendant of $v$ in $B(r,i)\setminus B(r,i-1)$. If there is such a vertex $w$ that has a forward edge with pink on its list, then we pick this vertex and colour that edge pink. Otherwise, we choose any such $w$ and pick an arbitrary colour (say red) for any of its forward edges. Then, we colour all the remaining uncoloured edges in $B(r,i)$ with arbitrary colours, such that:
\begin{itemize}
\item if $w$ has a red forward edge, then the colour red is not used on the forward edges of the vertices in $B(r,i)\setminus B(r,i-1)$, and
\item if $w$ has a pink forward edge, then we do not use the colour pink, and
\item if $w$ has a red forward edge, then each vertex in $B(r,i)$ except $r$ has at most one pink forward edge.
\end{itemize}

After these steps, $w$ is the only vertex in a distance $d(r,w)$ from $r$ with a pink (or red) forward edge. Therefore, $w$ is fixed, and so are all the vertices between $r$ and $w$ (including $v$). Since $\Delta\ge 3$, we did not create vertices with all incident pink edges, apart from $r$. Repeating these steps, we fix all the vertices of $G$. 



\textbf{Case 4.} $G$ is not regular, and the degree of every vertex of $G$ is in $\{1,\Delta\}$. We consider three subcases:

\textbf{Case 4a.} $G$ is finite. Then $G$ contains either a central vertex or a central edge.

If $G$ has a central vertex $r$, then, as $G$ is not a symmetric tree, $G-r$ must contain two rooted subtrees which are not isomorphic. We colour the edges incident to $r$ with distinct colours, except possibly two edges to two non-isomorphic subtrees. For the remaining edges, we use a standard colouring.

If $G$ has a central edge $xy$, we choose an arbitrary colour for that edge. Since $G$ is not a bisymmetric tree, among all the rooted subtrees of $G-e$, there must be two which are non-isomorphic. The roots of these subtrees are either the neighbours of the same end-vertex of the central edge, say $x$, or of two different end-vertices of the central edge. In both cases, we can colour the remaining edges incident to $y$ with different colours, and the same with $x$ (possibly using the same colour on the edges to the non-isomorphic subtrees), so that the palettes of $x$ and $y$ are different. Then we can continue with a standard colouring. 


\textbf{Case 4b.} $G$ contains a ray  but not a double ray.

Let $r$ be any non-leaf vertex on the unique ray of $G$. All but one subtree of $r$ must be finite, since otherwise $G$ would have a double ray. We colour all the edges from $r$ to its finite subtrees with different colours, and we choose any colour, say pink, for the last edge incident to $r$. Then, we complete this colouring to a standard colouring, with the additional condition that any forward edge to an infinite subtree has a colour other than pink. For any considered vertex, there will be at most one such forward edge, so this is possible, and it guarantees that $r$ is fixed.


\textbf{Case 4c.} $G$ contains a double ray.

Since $G$ has a leaf, there exists a vertex $r$ that lies on a double ray and has a finite subtree (and also two infinite ones). We try to choose different colours on the edges incident to $r$, and if it is impossible, then we repeat the colour on the edges to two non-isomorphic subtrees. Note that there is still an edge from $r$ to an infinite subtree with a different colour than the one to the finite subtree. 

Then, we continue with a standard colouring, with the additional condition that if for some vertex $r'$ we are forced to use the same palette as $r$, and there is an automorphism mapping $r'$ to $r$, then we use on the forward edge to the finite subtree a different colour than $r$ has. Note that the back edge of $r'$ leads to the subtree containing $r$, hence, to an infinite one. Therefore, the finite subtree, the existence of which is guaranteed by the automorphism, must be attached to one of the forward edges. 


\end{proof}
\section{Exceptional graphs}

For completeness, we append this short section about the locally finite graphs not covered by Theorems \ref{thm:cycle} and \ref{thm:trees}. We state the following theorems without proofs, as they are straightforward analogues of the proofs for the non-list distinguishing index, see \cite{KP}. 


\begin{thm}
    Let $G$ be the cycle of length $n$. Then $D'_l(G)=D'(G)=3$ if $n=3,4,5$, or $D'_l(G)=D'(G)=2$ otherwise. Moreover, for $n=3,4,5$ the only lists of length 2 which do not yield a distinguishing colouring are the identical ones. 
\end{thm}

\begin{thm}
    Let $G$ be the double ray, a symmetric tree, a bisymmetric tree, $K_4$, or $K_{3,3}$. Then $D_l'(G)=D'(G)=\Delta(G)$. Moreover, the only lists of length $\Delta-1$ which do not yield a distinguishing colouring are the identical ones, except for bisymmetric trees, where the central edge may have an arbitrary list (and the remaining ones must be identical). 
\end{thm}

\bibliographystyle{abbrv}
\bibliography{lit.bib}

\end{document}